\documentclass{jgcc}

\usepackage{hyperref}

 \usepackage{amssymb}
 \usepackage{amsmath}
 \usepackage{amsthm}

\newtheorem{remark}{Remark}
\newcommand{\bel}[1]{\begin{equation}\label{#1}}
\newcommand{\ee}{\end{equation}}

\newcommand{\LBA}{\left\{\begin{array}}
\newcommand{\EAR}{\end{array}\right.}

\newcommand{\prm}{\mathbb{P}}

\def\ovb{{\overline{b}}}
\def\ovc{{\overline{c}}}

\def\ovo{{\overline{0}}}

\def\ovv{{\overline{v}}}
\def\ovx{{\overline{x}}}
\def\ovy{{\overline{y}}}
\def\ovz{{\overline{z}}}

\def\ovw{{\overline{w}}}

\def\CA{{\mathcal A}}

\def\CG{{\mathcal G}}
\def\CH{{\mathcal H}}
\def\CI{{\mathcal I}}

\def\CL{{\mathcal L}}

\def\NP{{\mathbf{NP}}}
\def\coNP{{\mathbf{coNP}}}

\def\SIS{{\mathbf{SIS}}}
\def\ISIS{{\mathbf{ISIS}}}
\def\SIVP{{\mathbf{SIVP}}}

\def\SVP{{\mathbf{SVP}}}
\def\GAPSVP{{\mathbf{GapSVP}}}
\def\AGP{{\mathbf{AGWP}}}

\newcommand{\Set}[2]{\left\{\, #1 \;\middle|\; #2 \,\right\}}%set

\def\MN{{\mathbb{N}}}
\def\MZ{{\mathbb{Z}}}

\def\MR{{\mathbb{R}}}

\def\SSP{{\mathbf{SSP}}}

\def\CISE{{\mathbf{CISE}}}
\DeclareMathOperator{\Sph}{Sph}

\DeclareMathOperator{\SL}{{SL}}
\DeclareMathOperator{\PSL}{{PSL}}
\DeclareMathOperator{\GL}{{GL}}

\DeclareMathOperator{\Aut}{{Aut}}

\DeclareMathOperator{\Span}{{Span}}

\pdfoutput=1
\usepackage{lastpage}
\jgccdoi{16}{1}{3}{13555} 
\jgccheading{}{\pageref{LastPage}}{}{}{May.~6,~2024}{Jul.~5,~2024}{}    

\begin{document}

%%%%%%%%
%Title and author
\title[Constrained inhomogeneous spherical equations]{Constrained inhomogeneous spherical equations:\\ average-case hardness}
\author{Alexander Ushakov}
\address{Department of Mathematical Sciences, Stevens Institute of Technology, Hoboken NJ 07030, USA}\email{aushakov@stevens.edu}

\begin{abstract}
In this paper we analyze computational properties of the Diophantine problem
(and its search variant) for spherical equations 
$\prod_{i=1}^m z_i^{-1} c_i z_i = 1$ (and their variants)
over the class of finite metabelian groups 
$G_{p,n}=\MZ_p^n \rtimes \MZ_p^\ast$,
where $n\in\MN$ and $p$ is prime.
We prove that the problem of finding solutions for certain 
constrained spherical equations is computationally hard on average
(assuming that some lattice approximation problem is hard in the worst case).
\end{abstract}

\maketitle

\section{Introduction}

The main goal of this paper is to create bridges
between problems of computational group theory
(namely the problem of finding a solution for a spherical equation)
and assumptions of lattice-based cryptography.

Modern uses of lattices in the design of cryptographic primitives started in 1996 with the paper \cite{Ajtai-1996},
where M. Ajtai introduced a class of random problems
hard for probabilistic polynomial-time (PPT) algorithms
on any cryptographically non-negligible sets of instances
unless certain lattice approximation problems
(such as $\SIVP_\gamma$, see Section \ref{se:lattices})
can be solved efficiently in the worst case (on every input).
Since then Ajtai's construction was studied extensively;
it was improved in different ways, various interesting
applications were found, and it gained a lot of popularity,
see \cite{Peikert:2016} for a survey.
Nowadays, lattice-based cryptography appears to be the most
promising branch of post-quantum cryptography.

%\subsection{NIST post-quantum cryptography project}
%In 2017, NIST began a process to develop quantum resistant counterparts 
%to existing encryption standards.
%Almost $70$ proposals were considered.
%In 2023, after three rounds of a ``competition'' a
%four candidate algorithms have been selected for standardization. 
%NIST will recommend two primary algorithms to be implemented for most use cases:
%\begin{itemize}
%\item 
%CRYSTALS-KYBER (key-establishment, LWE-based),
%\item 
%CRYSTALS-Dilithium (digital signature, LWE-based).
%\end{itemize}
%Two additional schemes will be standardized:
%\begin{itemize}
%\item 
%FALCON (digital signature, LWE-based),
%\item 
%SPHINCS+ (digital signature, hash-based).
%\end{itemize}
%LWE is an abbreviation for the \emph{learning with errors problem},
%which average-case hardness was demonstrated by Regev in \cite{Regev:2005}
%assuming the worst case hardness of several lattice approximation
%problems such as $\SIVP_\gamma,\GAPSVP_\gamma, \BDD_\gamma$,
%discussed in Section \ref{se:lattices}.

\subsection{Equations in groups}

Let $F = F(Z)$ denote the free group on countably many generators 
$Z = \{z_i\}_{i=1}^\infty$. For a group $G$, an 
\emph{equation over $G$ with variables in $Z$} is an equality of 
the form $W = 1$, where $W \in F\ast G$. If 
$W = z_{i_1}^{\varepsilon_1}g_1\cdots z_{i_k}^{\varepsilon_k} g_k$, 
with $z_{i_j}\in Z$, $\varepsilon_j=\pm 1$, and $g_j\in G$,
then we refer to $\{z_{i_1},\ldots,z_{i_k}\}$ as the set of \emph{variables} and 
to $\{g_1,\ldots,g_k\}$ as the set of \emph{constants} (or \emph{coefficients}) of $W$. 
We occasionally write $W(z_1,\ldots,z_k)$ or $W(z_1,\ldots,z_k;g_1,\ldots,g_k)$ 
to indicate that the variables in $W$ are precisely $z_1,\ldots,z_k$ and (in the latter case)
the constants are precisely $g_1,\ldots, g_k$. 
A \emph{solution} for an equation $W(z_1,\ldots,z_k)=1$ over $G$ is
an assignment for variables $z_1,\ldots,z_k$ that makes $W=1$ true. 

In this paper we assume that $G$ comes equipped with a fixed generating set $X$
and elements of $G$ are given as products of elements of $X$ and their inverses.
This naturally defines the length (or size) of the equation $W=1$ as
the length of its left-hand side $W$.

\medskip\noindent
\textbf{The Diophantine problem} over a group $G$ for a class of equations $C$ 
is an algorithmic question to decide whether a given equation $W=1$
in $C$ has a solution. 

\medskip
By definition, the Diophantine problem is a \emph{decision} problem 
(a yes/no question).
Additionally, one can study the corresponding \emph{search} problem that requires
to find a solution for $W=1$, provided one exists.

\subsection{Constrained equations}

A \emph{constrained equation} over a group $G$ 
is an equation $W(z_1,\dots,z_n)=1$
equipped with a set $Z \subseteq G^n$.
The Diophantine problem for a constrained equation requires
to decide whether the following question has a solution or not:
$$
\left\{
\begin{array}{l}
W(z_1,\dots,z_n)=1,\\
(z_1,\dots,z_n)\in Z.
\end{array}
\right.
$$
In this paper we work with a set of the form $Z=Z_1\times\dots\times Z_n$, 
in which case every variable 
$z_i$ is constrained individually by $Z_i$.

Recently, constrained equations in groups have attracted significant attention, 
with most research focusing on free groups, hyperbolic groups, right-angled Artin 
groups, virtually abelian groups, and certain classes of finite groups.
For a survey of recent results in this area see \cite{Ciobanu:2024}.

% Constrained equations in groups have received very little attention;
% only certain equations in free groups have been studied.
% For instance, in \cite{Diekert:1998} V. Diekert demonstrated decidability 
% and $\PSPACE$-hardness of the following problem: given an arbitrary equation
% $$
% w(x_1,\dots,x_n;a_1,\dots,a_m)=1
% $$
% in a free group $F_m$ and a list $H_1,\dots,H_n$ of regular subsets of $F_m$,
% decide whether there exists a solution of this equation satisfying
% $x_i\in H_i$.
% V. Durnev showed in \cite{Durnev:2019} that for a free group $F=F(a,b)$
% there exists an equation of type
% $$
% w(x^k,x_1,\dots,x_n) = [a,b]
% $$
% for which there is no algorithm to recognize if, for a given $k$, it has a solution satisfying $x_1,\dots,x_n\in[F_2,F_2]$.

\subsection{Spherical equations}

One class of equations over groups that has generated much interest 
is the class of \emph{quadratic} equations: equations where each 
variable $z$ appears exactly twice (as either $z$ or $z^{-1}$). 
It was observed in the early 80's \cite{Culler:1981,Schupp:1980} that such equations have an affinity with the theory of compact surfaces (for instance, via their associated van Kampen diagrams).
This geometric point of view sparked the initial interest in their study and
 has led to many interesting results,
particularly in the realm of quadratic equations over free groups: 
solution sets were studied  in \cite{Grigorchuk-Kurchanov:1992}, 
$\NP$-completeness was proved in 
\cite{Diekert-Robson:1999,Kharlampovich-Lysenok-Myasnikov-Touikan:2010}.
Systems of quadratic equations
played an important role in the study of the first order theory of free groups 
(Tarski problem, \cite{Kharlampovich_Myasnikov:1998(1)}).
quadratic equations in various classes of (infinite) groups
such as hyperbolic groups 
(solution sets described in \cite{Grigorchuk-Lysenok:1992},
$\NP$-complete by \cite{Kharlampovich-Taam:2017}),
the first Grigorchuk group
(decidability proved in \cite{Lysenok-Miasnikov-Ushakov:2016},
commutator width computed in \cite{Bartholdi-Groth-Lysenok:2022}),
free metabelian groups
($\NP$-hard by \cite{Lysenok-Ushakov:2015},
in $\NP$ for orientable equations  
by~\cite{Lysenok-Ushakov:2021}),
metabelian Baumslag--Solitar groups
($\NP$-complete by \cite{Mandel-Ushakov:2023}), etc.

We say that equations $W = 1$ and $V=1$ are \emph{equivalent} 
if there is an automorphism $\phi\in \Aut(F\ast G)$ such that 
$\phi$ is the identity on $G$ and $\phi(W) = V$. It is a well 
known consequence of the classification of compact surfaces 
that any quadratic equation over $G$ is equivalent, via an 
automorphism $\phi$, computable in time $O(|W|^2)$, to an equation 
in exactly one of the following three \emph{standard forms} (see 
\cite{Comerford_Edmunds:1981,Grigorchuk-Kurchanov:1992}):
\begin{align}
\prod_{j=1}^m z_j^{-1} c_j z_j&=1 &m\ge 1,\label{eq:spherical}\\
\prod_{i=1}^g[x_i,y_i]\prod_{j=1}^m z_j^{-1} c_j z_j&=1 &g\geq 1, m\geq 0, \label{eq:orientable}\\
\prod_{i = 1}^g x_i^2\prod_{j=1}^m z_j^{-1} c_j z_j&=1 &g \geq 1, m\geq 0.\label{eq:nonorientable}
\end{align}
The number $g$ is the \emph{genus} of the equation, and both $g$ and $m$ 
(the number of constants) are invariants. 
The standard forms are called, respectively, 
\emph{spherical}, \emph{orientable of genus $g$}, and 
\emph{non-orientable of genus $g$}.

In this paper we investigate spherical equations in finite groups. 
We say that the equation \eqref{eq:spherical} is a \emph{homogeneous}
form of a spherical equation. 
For $m\in\MN$ define the set of all homogeneous equations
with $m$ conjugates by
$$
\Sph_m=\Set{\prod_{j=1}^m z_j^{-1} c_j z_j = 1}{c_1,\dots,c_m\in G}
\mbox{ and }
\Sph = \bigcup_{m=1}^\infty \Sph_m.
$$
It is easy to show that an equation of the form
\begin{equation}\label{eq:spherical-in}
\prod_{j=1}^m z_j^{-1} c_j z_j = c
\end{equation}
is equivalent to a spherical equation.
We say that an equation \eqref{eq:spherical-in}
is an \emph{inhomogeneous} form of a spherical equation.

Notice that spherical equations naturally generalize fundamental (Dehn)
problems of group theory, 
as solving equations from $\Sph_1$ 
is the same as solving the word problem
and solving equations from $\Sph_2$ 
is the same as solving the conjugacy problem.
The complexity of solving equations in finite groups has been first studied by 
Goldmann and Russell \cite{Goldmann-Russell:2002} showing that the Diophantine problem 
in a fixed finite nilpotent group can be decided in polynomial time, while it is 
$\NP$-complete in every finite non-solvable group. 
For some recent results, see \cite{Foldvari-Horvath:2020,Idziak-Kawalek-Krzaczkowski-Weiss:2022,Mattes-Ushakov-Weiss:2024}. 

\subsection{Groups under consideration}

Consider the group 
$\MZ_p^n=\Set{(x_1,\ldots,x_n)}{x_i\in\MZ_p}$.
The group of units $\MZ_p^\ast$ acts on $\MZ_p^n$
by (scalar) multiplication
$$
(x_1,\ldots,x_n)
\ \stackrel{\alpha}{\mapsto}\ 
(\alpha x_1,\ldots,\alpha x_n),
$$
where $\alpha\in\MZ_p^\ast$ and $(x_1,\ldots,x_n)\in\MZ_p^n$.
The semidirect product $G = \MZ_p^n\rtimes \MZ_p^\ast$ 
is a set of pairs $(\ovx,\alpha)$ equipped with the binary operation
$$
(\ovx,\alpha)
(\ovy,\beta) = (\ovx+\alpha\ovy,\alpha\cdot\beta),
$$
with the identity $(\ovo,1)$.
The following useful formulae are used throughout the paper
without referencing:
\begin{align*}
(\ovx,\alpha)^{-1} 
&= 
(-\alpha^{-1}\ovx,\alpha^{-1}),\\
(\ovx,\alpha)^{-1}
(\ovy,\beta)
(\ovx,\alpha)
&=
(\alpha^{-1}((\beta-1)\ovx+\ovy),\beta).
\end{align*}

\subsection{Model of computation}

We assume that all computations are performed on a random access machine.
Elements of $\MZ_p$ are given in binary as bit-strings of length
$\lceil \log_2(p) \rceil$.
Elements of $\MZ_p^n\rtimes \MZ_p^\ast$ are given 
as $n+1$-tuples of elements from $\MZ_p$.
If $f(n)$ is $O(T(n)\cdot n^{\varepsilon})$ for every $\varepsilon>0$, then
we say that a function $f(n)$ is ``nearly $T(n)$'' 
and write $f(n)$ is $\Tilde{O}(T(n))$.
Operations in $\MZ_p$ have the following complexity:
\begin{itemize}
\item 
Addition and subtraction can be done in $O(\lceil \log_2(p) \rceil)$
time in a straightforward way.
\item
Multiplication can be done in 
$\Tilde{O}(\lceil \log_2(p) \rceil)$ time using fast Fourier transform.
\item
Computing the multiplicative inverse of a unit modulo $p$ can be done in 
$O(\lceil \log_2(p) \rceil^2)$ time using the extended Euclidean algorithm.
\end{itemize}

\section{Preliminaries: subset sum problem in groups}

Here we review several useful definitions of discrete optimization in groups.
Let $G$ be a group generated by a finite set 
$X=\{x_1,\ldots,x_n\}\subseteq G$. Elements in $G$ can be expressed
as products of the generators in $X$ and their inverses.
Hence, we can state the following combinatorial problem.

\medskip
\noindent{\bf The subset sum  problem $\SSP(G,X)$\index{$\SSP(G,X)$}:}
Given $g_1,\ldots,g_k,g\in G$ decide if
  \begin{equation} \label{eq:SSP-def}
  g = g_1^{\varepsilon_1} \cdots g_k^{\varepsilon_k}
  \end{equation}
for some $\varepsilon_1,\ldots,\varepsilon_k \in \{0,1\}$.

\medskip
By \cite[Proposition 2.5]{Miasnikov-Nikolaev-Ushakov:2015},
computational properties of $\SSP$ do not depend on the choice of a finite generating set $X$
and, hence, the problem can be abbreviated as $\SSP(G)$.
The same paper, \cite{Miasnikov-Nikolaev-Ushakov:2015},
provides a variety of examples
of groups with $\NP$-complete (or polynomial time) subset sum problems.

Consider the infinitely generated group $\MZ_3^\omega$ whose elements 
can be formally viewed as functions $f\colon\MN\to\MZ_3$ with finite support.
For algorithmic purposes, we assume that elements of $\MZ_3^\omega$
are encoded by finite ternary strings 
(as in \cite[Section 4]{Miasnikov-Nikolaev-Ushakov:2014b}).
\cite[Proposition 2.1]{Nikolaev-Ushakov:2020} proves 
that $\SSP(\MZ_3^\omega)$ is $\NP$-complete, 
which can be reformulated as follows.

\begin{prop}[{{\cite[Proposition 2.1]{Nikolaev-Ushakov:2020}}}]\label{pr:SSP_Z3_omega}
For every $m\ge 3$, $\SSP$ is $\NP$-complete for the class 
of finite groups $\{\MZ_m^n\}_{n=1}^\infty$.
\end{prop}

\section{Preliminaries: lattice problems}

\subsection{Lattices}\label{se:lattices}

Here we review several definitions of the theory of lattices.
Recall that a set of points $S\subseteq \MR^n$ is 
\emph{discrete} if every point $x\in\MR^n$
has an $\varepsilon$-neighborhood
that contains $x$ only.
A \emph{lattice} is a discrete subgroup of $\MR^n$.
An \emph{integer lattice} is a subgroup of $\MZ^n$.
We discuss integer lattices only.
We say that $\CL\le \MZ^n$ is a \emph{full-rank} lattice 
if the dimension of the corresponding vector space $\Span(\CL)$
is $n$ (we can simply write $\dim(\CL)=n$).
We typically assume that $\CL$ is a full-rank lattice.

The \emph{minimum distance} of a lattice $\CL\le\MR^n$
is the length of a shortest nonzero lattice vector
$$
\lambda_1(\CL) = \min_{\ovv\in\CL\setminus\{\ovo\}} \|\ovv\|.
$$
More generally, for $i=1,\ldots,n$, the $i$th
\emph{successive minimum} of $\CL$ is 
$$
\lambda_i(\CL) = \min\left\{r\mid \dim(\Span(B(r)\cap\CL))\ge i\right\},
$$
which is the minimum radius of a ball that contains at least $i$
linearly independent points.
Below we recall several computational problems for lattices.
Some of them are of direct importance to lattice-based cryptography
(e.g., $\SIVP_\gamma$ and $\GAPSVP_\gamma$) and others have
more historical importance.

\medskip\noindent
\textbf{Shortest vector problem}, $\SVP$.
Given a basis of a lattice $\CL\le \MZ^n$,
find a shortest nonzero vector $\ovv \in \CL$, i.e., 
a vector satisfying $\|\ovv\| = \lambda_1(\CL)$.

\medskip\noindent
\textbf{Approximate shortest vector problem}, $\SVP_\gamma$.
Given a basis of a lattice $\CL\le \MZ^n$,
find a nonzero vector $\ovv \in \CL$ satisfying
$\|\ovv\|\le \gamma\cdot \lambda_1(\CL)$.

\medskip\noindent
\textbf{Decisional approximate SVP}, $\GAPSVP_\gamma$.
Given a basis of a lattice $\CL$,
where either $\lambda_1(\CL)\le 1$ or $\lambda_1(\CL)>\gamma$,
decide which is the case.

\medskip\noindent
\textbf{Shortest independent vector problem}, $\SIVP$.
Given a basis of a full-rank lattice $\CL\le \MZ^n$,
find $n$ linearly independent vectors $\ovv_1,\dots,\ovv_n\in \CL$
satisfying $\max \|\ovv_i\| \le \lambda_n(\CL)$.

\medskip\noindent
\textbf{Approximate shortest independent vectors problem}, $\SIVP_\gamma$.
Given a basis of a full-rank lattice $\CL\le \MZ^n$,
find $n$ linearly independent vectors $\ovv_1,\dots,\ovv_n\in \CL$
satisfying
$$
\max \|\ovv_i\| \le \gamma(n) \cdot \lambda_n(\CL).
$$

\medskip
Formally, we say that $\SIVP_{\gamma}$ is hard in the worst case
for probabilistic polynomial-time (PPT) algorithms,
if for every PPT algorithm $\CA$
and for every $n\in\MN$ there is an $n$-dimensional basis
$B_{n,\CA}=\{\ovb_1,\dots,\ovb_n\}\subseteq \MZ^n$  satisfying
\begin{equation}\label{eq:SIVP-hardness}
\Pr[\CA \mbox{ solves } \SIVP_{\gamma} \mbox{ for } B_{n,\CA}] 
\ \ \mbox{ is }\ \  o(n^{-d}),
\end{equation}
for every $d>0$, where the probability is taken over the coin-tosses of $\CA$.

%\medskip\noindent
%\textbf{Bounded distance decoding problem}, ($\BDD_\gamma$).
%Given a basis of an $n$-dimensional lattice $\CL$, 
%and a target point $\ovt \in \MR^n$ satisfying
%$d(\ovt,\CL)<d=\tfrac{\lambda_1(\CL)}{2\gamma}$,
%find the unique lattice vector $\ovv\in\CL$ 
%such that $d(t,v)<d$.
%\sasha{Suggestion: delete $\BDD$.}
% 
%\medskip
%By definition $\BDD_\gamma$ is a promise problem:
%the target $\ovt$ is promised to be
%$\tfrac{\lambda_1(\CL)}{2\gamma}$-close to the lattice $\CL$. 
%This promise and the uniqueness of the solution distinguish $\BDD_\gamma$ from 
%the approximate closest vector problem $\CVP_\gamma$,
%wherein the target $\ovt$ can be an arbitrary point.

The above lattice problems have been intensively studied and appear to be
intractable, except for very large approximation factors $\gamma(n)$.
Known polynomial-time algorithms like the
Lenstra--Lenstra--Lov\'asz \cite{Lenstra-Lenstra-Lovasz:1982}
and its descendants obtain only slightly subexponential approximation factors 
for all the above problems.
Known algorithms that obtain polynomial
approximation factors, such as \cite{Kannan:1983, Ajtai-Kumar-Sivakumar:2001, Micciancio-Voulgaris:2010, Aggarwal-Dadush-Regev-Stephens-Davidowitz:2015},
either require superexponential time, or exponential time and space.

Many lattice problems are $\NP$-hard, even to approximate to within various sub-polynomial $n^{o(1)}$ approximation factors.
However, such hardness is not of any direct consequence to cryptography, 
since lattice-based cryptographic constructions so far rely on polynomial
approximation problems factors $\gamma(n)\ge n$.
Indeed, there is evidence that for factors 
$\gamma(n) \ge \sqrt{n}$, the lattice
problems relevant to cryptography are not $\NP$-hard, because they lie in 
$\NP \cap \coNP$ \cite{Goldreich-Goldwasser:2000, Aharonov-Regev:2005}.
For a survey on lattice-based cryptography see \cite{Peikert:2016}.

\subsection{Short integer solution problem}

In general, the \emph{short integer solution} problem
can be formulated as follows.

\medskip\noindent
\textbf{Short integer solution ($\SIS$) problem.}
For a given matrix $A\in\MZ_p^{n\times m}$,
find $\ovx\in\MZ_p^m$ satisfying
\begin{equation}\label{eq:SIS}
\left\{
\begin{array}{l}
A\ovx\equiv_p \ovo,\\
\ovx \mbox{ is short and nontrivial},\\
\end{array}
\right.
\end{equation}
assuming that a solution exists.

\medskip
In other words, $\SIS$ requires to find 
a short and nontrivial solution for a homogeneous system of 
linear congruences $A\ovx\equiv_p \ovo$. 
In the original paper \cite{Ajtai-1996}, 
$\ovx$ was called short if $\|\ovx\|_2\le n$
(can be relaxed to $\|\ovx\|_2\le poly(n)$).
In \cite{Goldreich-Goldwasser-Halevi:1997}, 
instead of $\|\ovx\|_2\le n$, 
the following two constraints on $\ovx$ were discussed:
\begin{itemize}
\item 
$\ovx\ne 0$ and $x_i\in \{0,1\}$;
\item 
$\ovx\ne 0$ and $x_i\in \{-1,0,1\}$.
\end{itemize}
We denote the corresponding versions of $\SIS$ by
$\SIS_{\{0,1\}}$ and $\SIS_{\{-1,0,1\}}$.
To discuss the average case complexity of $\SIS$ we use 
the uniform distribution on its instances.

\medskip\noindent
\textbf{Randomized $\SIS$ problem}.
For a matrix  $A\in \MZ_p^{n\times m}$
sampled uniformly randomly,
find $\ovx\in\MZ_p^m$ such that
$A\ovx=\ovo$, and
$\ovx$ is short and nontrivial.

\begin{thm}[{{Goldreich, Goldwasser, Halevi, \cite[Theorem 1]{Goldreich-Goldwasser-Halevi:1997}, cf. Ajtai, \cite[Theorem 1]{Ajtai-1996}}}]
\label{th:Ajtai}
Suppose that a PPT algorithm $\CA$ solves
the randomized $\SIS_{\{0,1\}}$ problem 
(or $\SIS_{\{-1,0,1\}}$ problem)
with parameters $n,m,p$ satisfying
\begin{equation}\label{eq:mnp-conditions}
n\log(p) < m < \tfrac{p}{2n^4}
\ \mbox{ and }\ 
p=O(n^c) \mbox{ for some }c>0,
\end{equation}
with probability at least $n^{-c_0}$ for some fixed constant $c_0>0$,
where the probability is taken over the choice of 
the instance as well as the coin-tosses of $\CA$.
Then there is a PPT algorithm
that solves 
$\GAPSVP_{\gamma=pn^6}$ and $\SIVP_{\gamma=pn^6}$ (among others)
on every $n$-dimensional lattice with probability at least $1-2^{-n}$.
\footnote{For better bounds and approximation factors of $\gamma$ see 
\cite{Micciancio-Regev:2007}.}
\end{thm}

In a similar way we can define deterministic and randomized 
\emph{inhomogeneous short integer solution} (ISIS) problem. 

\medskip\noindent
\textbf{Randomized $\ISIS$ problem}.
For a matrix  $A\in \MZ_p^{n\times m}$ and a vector $\ovy\in\MZ_p^n$
sampled uniformly randomly,
find $\ovx\in\MZ_p^m$ such that
$A\ovx=\ovy$, and
$\ovx$ is short.

\begin{cor}
The statement of Theorem \ref{th:Ajtai} holds for 
the randomized $\ISIS_{\{0,1\}}$ problem.
\end{cor}

\begin{proof}
Indeed, suppose that there is a PPT algorithm $\CA$ that solves
$\ISIS_{\{0,1\}}$ problem with probability at least $n^{-c_0}$.
Then for an instance $A$ of $\SIS_{\{0,1\}}$ 
with columns $\ovv_1,\dots,\ovv_m$ 
\begin{enumerate}
\item 
``guess'' a nonzero index $i$ in a solution $\ovx\in\{0,1\}^m$,
\item 
form an instance $(A',-\ovv_i)$, where
$A'$ is obtained by deleting the $i$th column from $A$,
\item 
solve $(A',-\ovv_i)$ using $\CA$.
\end{enumerate}
This gives a PPT algorithm that solves $\SIS_{\{0,1\}}$
with probability at least $n^{-c_0}$.
\end{proof}

\section{Spherical equations over $G_{p,n}$}

In this section we analyse complexity of solving a spherical equation
\eqref{eq:spherical}
over $G_{p,n}$, with constants $c_i=(\ovc_i,\beta_i)\in G_{p,n}$ 
and unknowns $z_1,\dots, z_m$.

\begin{lem}\label{le:main-conditions}
$z_i=(\ovz_i,\alpha_i)$ satisfy 
\eqref{eq:spherical}
$\ \ \Leftrightarrow\ \ $
the following two conditions are satisfied:
\begin{itemize}
\item[(S1)]
$\beta_1\cdots\beta_m=1$;
\item[(S2)]
$\sum_{i=1}^m
B_i\alpha_i^{-1}
((\beta_i-1) \ovz_i+\ovc_i)
\equiv_p \ovo$,
where $B_i=\beta_1\cdots\beta_{i-1}$.
\end{itemize}
\end{lem}

\begin{proof}
Straightforward verification.
\end{proof}

\subsection{Spherical equations over $G_{p,n}$: generic-case}

Here, we investigate generic-case hardness for spherical equations,
i.e., hardness of a ``typical'' equation, see \cite[Chapter 10]{MSU_book:2011} 
for basic definitions of generic-case complexity.

\begin{prop}\label{pr:beta-ne-one}
If $\beta_i\ne 1$ for some $i$, then \eqref{eq:spherical} has a solution 
and a solution can be found in polynomial time.
\end{prop}

\begin{proof}
It takes nearly linear time to compute $B_1,\dots,B_m$ and check 
the condition (S1).
If $\beta_i\ne 1$, then the following assignment:
\begin{itemize}
\item 
$\alpha_1=\dots=\alpha_m=1$,
\item 
$\ovz_j=\ovo$ for $j\ne i$,
\item 
$
\ovz_i
\ \ \equiv_p\ \ 
\tfrac{1}{B_i (\beta_i-1)}
\sum_{j=1}^m B_j \ovc_j
$
\end{itemize}
is a solution that can be computed in polynomial time.
\end{proof}

Next, we claim that the property 
$\exists i\ \beta_i\ne 1$ is strongly generic, 
i.e., a typical equation satisfies this property.
Since the problem involves three parameters, $n,m$, and $p$, we use the following 
stratification for the set of instances of the uniform problem. For $s\in\MN$ 
define a set of pairs
$$
\CI_s = \Set{(E_m,G_{p,n})}{E_m\in \Sph_m(G_{p,n}),\ m,n,p\le s}
$$
equipped with the uniform distribution.
Then the following holds:
\begin{align*}
|G_{p,n}| &= p^n\cdot (p-1),\\
|\Sph_m| &= (p^n\cdot (p-1))^m,\\
|\CI_s| &= \sum_{p,n,m\le s} (p^n\cdot (p-1))^m.
\end{align*}

\begin{lem}\label{le:m-generic}
$\Pr\Set{(E_m,G_{p,n})\in \CI_s}{m\ge s/2} \to 1$ 
exponentially fast as $s\to\infty$.
\end{lem}

\begin{proof}
For any fixed $p\ge 2,n\ge 1$ we have
$$
V_s=\sum_{m=0}^s 
(p^n\cdot (p-1))^m 
= \frac{(p^n\cdot (p-1))^{s+1}-1}{p^n\cdot (p-1)-1},
$$
which implies that
$$
\frac{V_{s/2}}{V_s} = 
\frac{(p^n\cdot (p-1))^{s/2+1}-1}{(p^n\cdot (p-1))^{s+1}-1}
\le
\frac{1}{(p^n\cdot (p-1))^{s/2}} 
\le
\frac{1}{2^{s/2}} 
$$
that converges to $0$ exponentially fast. The obtained bound
$\frac{1}{2^{s/2}}$ does not depend on $p$ or $n$.
Therefore, the bound works on the whole $\CI_s$.
\end{proof}

\begin{lem}\label{le:beta-generic}
$\Pr\Set{(E_m,G_{p,n})\in \CI_s}{\exists i\ \beta_i\ne 1} \to 1$ 
exponentially fast as $s\to\infty$.
\end{lem}

\begin{proof}
By Lemma \ref{le:m-generic}, we may assume that $m\ge s/2$, in which case 
$$
\Pr\Set{(E_m,G_{p,n})\in \CI_s}{\exists i\ \beta_i\ne 1} \ge 1-\frac{1}{(p-1)^{s/2}},
$$ 
which converges to $1$ exponentially fast as $s\to\infty$.
\end{proof}

\begin{cor}
The Diophantine (decision) problem for spherical equations over $\{G_{p,n}\}$
is decidable in strongly generically linear time. 
The corresponding search problem can be solved in
strongly generically polynomial time. 
\end{cor}

\begin{proof}
Follows from Proposition \ref{pr:beta-ne-one}
and Lemma \ref{le:beta-generic}
\end{proof}

Since spherical equations in which at least one coefficient 
$c_i=(\ovc_i,\beta_i)$ satisfies $\beta_i\ne 1$
are computationally easy,
we put a restriction on the coefficients that we use in spherical equations. 
Define the set
\begin{equation}\label{eq:C-p-n}
C_{p,n}=\Set{(\ovc,1)}{\ovc\in \MZ_p^n}.
\end{equation}
From now on we assume that all $c_i\in C_{p,n}$.
In that case condition (S1) of Lemma \ref{le:main-conditions}
is trivially satisfied and condition (S2) translates into
the following:
\begin{equation}\label{eq:nonzero-combination-problem}
\exists \alpha_i\in\MZ_p^\ast
\ \ \mbox{ s.t. }\ 
\sum_{i=1}^m \alpha_i^{-1} \ovy_i \equiv_p \ovo.
\end{equation}
The obtained condition defines a homogeneous system of linear
congruences that does not allow zero values for unknowns $\alpha_i$,
which makes the problem nontrivial (we conjecture it is $\NP$-hard).

\subsection{Spherical equations over $G_{p,n}$: worst case}
\label{se:sph-np-hard}

Here we investigate the worst case complexity for spherical equations.
For $\ovw\in\MZ_3^n$ define 
$$
c_\ovw = (\ovw,1)\in \MZ_{3}^n\rtimes \MZ_3^\ast.
$$
Consider an instance $I$ of $\ovv_1,\dots,\ovv_m,\ovv \in \MZ_3^n$ of $\SSP$.
Recall that $I$ is a positive instance if 
$\varepsilon_1\ovv_1+\dots+\varepsilon_m\ovv_m=\ovv$ for some 
$\varepsilon_1,\dots,\varepsilon_m\in\{0,1\}$.
For $I$ define $\ovv'=\ovv+\sum \ovv_i$ and
the spherical equation $E_I$ as
\begin{equation}\label{eq:E_I}
\prod_i^m z_i^{-1}c_{\ovv_i} z_i = c_{\ovv'}.
\end{equation}

\begin{prop}\label{pr:equations-worst-hard}
$I$ is a positive instance of $\SSP$
$\ \ \Leftrightarrow\ \ $
$E_I$ has a solution.
\end{prop}

\begin{proof}
By design we have
\begin{align*}
I \mbox{ is a positive instance of } \SSP
&\ \Leftrightarrow\ 
\ovv=\sum_{i=1}^m \varepsilon_i\ovv_i
\mbox{ for } \varepsilon_1,\dots,\varepsilon_m\in\{0,1\}\\
&\ \Leftrightarrow\ 
\ovv+\sum_{i=1}^m \ovv_i=\sum_{i=1}^m \beta_i\ovv_i
\mbox{ for } \beta_1,\dots,\beta_m\in\{1,2\}=\MZ_3^\ast\\
&\ \Leftrightarrow\ 
\ovv+\sum_{i=1}^m \ovv_i=\sum_{i=1}^m \alpha_i^{-1} \ovv_i
\mbox{ for } \alpha_1,\dots,\alpha_m\in\{1,2\}=\MZ_3^\ast\\
&\ \stackrel{\ref{le:main-conditions}}{\Leftrightarrow}\ 
E_I
\mbox{ has a solution.}
\end{align*}
\end{proof}

\begin{cor}\label{co:equations-worst-hard}
The Diophantine problem 
for spherical equations over $\MZ_{3,n}\rtimes \MZ_3^\ast$
is $\NP$-complete.
\end{cor}

\begin{proof}
By Proposition \ref{pr:equations-worst-hard}, $I\to E_I$
is a many-one polynomial-time reduction of $\SSP(\MZ_{3}^n)$
to spherical equations over $\SSP(\MZ_{3,n}\rtimes\MZ_3^\ast)$.
By Proposition \ref{pr:SSP_Z3_omega}, $\SSP$ is $\NP$-complete
for groups $\{\MZ_3^n\}_{n=1}^\infty$. 
Hence the result.
\end{proof}

\section{Constrained spherical equations: average-case hardness}
\label{se:cise}

In this section we discuss spherical equations over the 
class of groups $\CG=\{G_{p,n}\}$
with constraints on values of $z_i$'s. 
In full generality the problem can be formulated as follows.
Given a group $G\in\CG$, a spherical equation
$\prod_{i=1}^m z_i^{-1} c_i z_i = 1$
over $G$, and subsets $Z_1,\dots,Z_m \subseteq G$,
find $z_1,\dots,z_m\in G$ satisfying
$$
\left\{
\begin{array}{l}
\prod_{i=1}^m z_i^{-1} c_i z_i= 1\\
z_i\in Z_i
\end{array}
\right.
$$
or a similar \emph{inhomogeneous} form
$$
\left\{
\begin{array}{l}
\prod_{i=1}^m z_i^{-1} c_i z_i= c\\
z_i\in Z_i.
\end{array}
\right.
$$
The problem to decide if a given constrained inhomogeneous spherical equation
has a solution is abbreviated $\CISE$.

It is not hard to prove that $\CISE$ for groups $G_{p,n}$ 
is hard in the worst case.
In fact, for $p=3$, the proof of Proposition \ref{pr:equations-worst-hard}
works as is, with constraints $Z_i=\{(\ovo,1),(\ovo,2)\}$
(so chosen sets $Z_i$ effectively define no constraints).
Using these constraints we can easily achieve the same result 
for any odd value of $p$.

\begin{thm}
$\CISE$ is $\NP$-hard for groups $\{G_{p,n}\}_n$ for any fixed odd prime $p$.
\end{thm}

\begin{proof}
For a given instance $\ovv_1,\dots,\ovv_m,\ovv \in \MZ_p^n$, call it $I$,
of $\SSP(\MZ_p^n)$, as in the proof of Proposition \ref{pr:equations-worst-hard},
construct a constrained inhomogeneous spherical equation $E_I$ 
over $\MZ_p^n\rtimes\MZ_p^\ast$
$$
\left\{
\begin{array}{l}
\prod_{i=1}^m z_i^{-1}c_{\ovv_i} z_i = c_{\ovv'}\\
z_i = (\ovo,1) \mbox{ or } (\ovo,2^{-1})
\end{array}
\right.
$$
and observe that $I$ is a positive instance of $\SSP(\MZ_p^n)$
if and only if $E_I$ has a solution.
This gives a polynomial time reduction from 
$\SSP$ over groups $\{\MZ_p^n\}_{n=1}^\infty$
to decidability of constrained inhomogeneous spherical equations 
over $\MZ_p^n\rtimes\MZ_p^\ast$.
\end{proof}

Our next goal is to demonstrate the average-case hardness 
of constrained spherical equations. 
To achieve this, we randomize equations as follows.
First, we stratify parameters:
\begin{itemize}
\item 
the parameter $n\in\MN$ is the main independent parameter;
\item 
parameters $m$ and $p$ are functions of $n$ satisfying conditions
\eqref{eq:mnp-conditions}.
\end{itemize}
Then, for an arbitrary but fixed $n\in\MN$, consider 
the set of coefficients $C_{p,n}$
defined in \eqref{eq:C-p-n}
and a (finite) set of constrained inhomogeneous spherical equations 
\begin{equation}\label{eq:CSE-12}
\left\{
\begin{array}{l}
\prod_{i=1}^m z_i^{-1} c_i z_i= c\ \  \mbox{ with } c_i,c\in C_{p,n}\\
z_i=(\ovo,1)\mbox{ or }(\ovo,2^{-1}).
\end{array}
\right.
\end{equation}
The problem to find a solution of \eqref{eq:CSE-12} is called $\CISE_{\{1,2\}}$. 
Also, we want to study slightly relaxed equations 
\begin{equation}\label{eq:CSE-123}
\left\{
\begin{array}{l}
\prod_{i=1}^m z_i^{-1} c_i z_i= c\ \  \mbox{ with } c_i,c\in C_{p,n}\\
z_i=(\ovo,1),(\ovo,2^{-1}),\mbox{ or }(\ovo,3^{-1}).
\end{array}
\right.
\end{equation}
The problem to find a solution of \eqref{eq:CSE-123} is called 
$\CISE_{\{1,2,3\}}$.

\medskip\noindent
\textbf{Randomized constrained inhomogeneous spherical equation problem}.
Find a solution for a uniformly distributed system \eqref{eq:CSE-12}
(or \eqref{eq:CSE-123}).

\medskip
\begin{thm}\label{th:CISE-hard}
Suppose that a PPT algorithm $\CA$ solves
the randomized $\CISE_{\{1,2\}}$ problem 
(or $\CISE_{\{1,2,3\}}$ problem)
with parameters $n,m,p$ satisfying \eqref{eq:mnp-conditions}
with probability at least $n^{-c_0}$ for some fixed constant $c_0>0$,
where the probability is taken over the choice of 
the instance as well as the coin-tosses of $\CA$.
Then there is a PPT algorithm
that solves $\SIVP_{\gamma=pn^6}$ for any lattice
with overwhelming probability (e.g. $1-2^{-n}$).
\end{thm}

\begin{proof}
Consider an arbitrary instance of $(A,\ovy)$ of $\ISIS_{\{0,1\}}$.
Let $\ovv_1,\dots,\ovv_m$ be the system of columns of $A$. 
Consider the equation \eqref{eq:E_I},
where, as in Section \ref{se:sph-np-hard}, 
$c_\ovv=(\ovv,1)$ and $\ovv'=\ovy + \sum_{i=1}^m \ovv_i$.
Then for $\ovx=(x_1,\dots,x_m)\in\{0,1\}^m$ we have
\begin{align*}
\ovx
\mbox{ is a solution for $(A,\ovy)$}
&\ \ \Leftrightarrow\ \ 
\sum_{i=1}^m x_i \ovv_i =  \ovy\\
&\ \ \Leftrightarrow\ \ 
\sum_{i=1}^m (x_i+1)\ovv_i =  \ovy + \sum_{i=1}^m \ovv_i\\
&\ \ \Leftrightarrow\ \ 
\prod_{i=1}^m (\ovo,(x_i+1)^{-1})^{-1} c_{\ovv_i} (\ovo,(x_i+1)^{-1}) 
= c_{\ovv'}\\
&\ \ \Leftrightarrow\ \ 
\{z_i=(\ovo,(x_i+1)^{-1})\}_{i=1}^m
\mbox{ is a solution for \eqref{eq:CSE-12}.}
\end{align*}
This establishes a one-to-one correspondence between 
$\ISIS_{\{0,1\}}$ and $\CISE_{\{1,2\}}$ that
translates the uniform distribution for the instances of $\ISIS_{\{0,1\}}$
to the uniform distribution for the instances of $\CISE_{\{1,2\}}$
(the element $\ovv'\in C_{p,n}$ is uniformly distributed because $\ovy$
uniformly distributed).
A similar reduction can be used to establish
a one-to-one correspondence between 
$\SIS_{\{-1,0,1\}}$ and $\CISE_{\{1,2,3\}}$. 
\end{proof}

\begin{remark}
The one-to-one correspondence established in the proof of Theorem \ref{th:CISE-hard}
also implies that uniformly chosen instances of 
$\CISE_{\{1,2\}}$ and $\CISE_{\{1,2,3\}}$
have solutions with probability approaching $1$ as $n\to\infty$.
\end{remark}

\begin{remark}
Similar results can be proven for the class of symmetric groups 
$\{S_n\}_{n=1}^\infty$,
or any other class of groups that contain the groups $G_{p,n}$ in a ``compact'' way.
\end{remark}

\section{Spherical functions}

Let $G$ be a group.
For $\ovc=(c_1,\dots,c_m)\in G^m$ define a function $f_\ovc\colon G^m\to G$ by
$$
(z_1,\dots,z_m) 
\ \ \stackrel{f_\ovc}{\mapsto}\ \ 
(z_1^{-1} c_1 z_1)
\cdots
(z_m^{-1} c_m z_m).
$$
We call $f_\ovc$ a \emph{spherical function} because the problem of finding
an $f_\ovc$-preimage of $g\in G$ is equivalent to finding a solution 
for a spherical equation
$(z_1^{-1} c_1 z_1)
\cdots
(z_m^{-1} c_m z_m)=g
$.
Fix distinct $g_0,g_1\in G$ and define a function $H_\ovc\colon \{0,1\}^m\to G$ by
$$
(b_1,\dots,b_m)
\ \ \stackrel{H_\ovc}{\mapsto}\ \ 
(g_{b_1}^{-1} c_1 g_{b_1})
\cdots
(g_{b_m}^{-1} c_m g_{b_m}).
$$ 
called a \emph{$0/1$-spherical function} of length $m$.
Similarly, we can define $-1/0/1$-spherical functions of length $m$.

Consider the family of groups $\{G_{p,n}\}$ and
$g_0=(\ovo,1), g_1=(\ovo,2^{-1})\in G_{p,n}$, where $2^{-1}=\tfrac{1}{2}(p+1)$ is 
the multiplicative inverse of $2$ modulo an odd prime $p$.
As in Section \ref{se:cise}, $n$ is considered to be the main parameter;
the parameters $m$ and $p$ are functions of $n$ satisfying conditions 
\eqref{eq:mnp-conditions}.
Define a system of functions
$$
\CH_n=\Set{H_\ovc}{c_1,\dots,c_m\in C_{p,n}}
\mbox{ and }
\CH= \bigcup_{n=1}^\infty \CH_n.
$$

\begin{prop}\label{pr:one-way}
If $\SIVP_{\gamma=pn^6}$ is hard in the worst case, then
$\CH$ is a one-way function family.
\end{prop}

\begin{proof}
To prove that $\CH$ is a one-way function family it is sufficient to show that
for every PPT algorithm $\CA$ the sequence of values
$$
P_\CA(n)\ =\ \Pr[\CA(1^n,\ovc,H_\ovc(\ovx)) \in H_\ovc^{-1}(H_\ovc(\ovx))]
$$
converges to $0$ faster than every $1/poly(n)$,
where the probability is taken over
\begin{itemize}
\item 
uniform choices of $\ovc\in (C_{p,n})^m$,
\item 
uniform choices of $\ovx\in \{0,1\}^n$, and
\item 
the coin-tosses of $\CA$.
\end{itemize}
It is not difficult to check that 
for uniformly distributed $\ovc$ and $\ovx$,
the values of $H_\ovc(\ovx)\in C_{p,n}$ are distributed
nearly uniformly (the sequence of distributions on $C_{p,n}$ 
converges to the uniform distribution exponentially fast in terms of $n$) and
we may assume that $\CA$ deals with uniformly randomized $\CISE_{\{1,2\}}$.

Therefore, if it is not true that $P_\CA(n)$ is $o(n^{-d})$
for some PPT algorithm $\CA$ and some $d>0$, then 
$P_\CA(n) \ge c\cdot n^{-d}$ (for some $c>0$) satisfied
for infinitely many indices $n$. 
By Theorem \ref{th:CISE-hard} that means that there is a PPT algorithm 
solving $\SIVP_{\gamma=pn^6}$ with overwhelming probability 
on every lattice for infinitely many dimensions $n$, 
which contradicts the assumption \eqref{eq:SIVP-hardness}. 
\end{proof}

\begin{prop}
If $\SIVP_{\gamma=pn^6}$ is hard in the worst case, then
$\CH$ is a collision-free function family.
\end{prop}

\begin{proof}
To prove that $\CH$ is a collision-free hash function family
it is sufficient to show that for every PPT algorithm $\CA$ the sequence of values
$$
P_\CA(n)\ =\ \Pr[(\ovx,\ovy)=\CA(1^n,\ovc)\ \&\ H_\ovc(\ovx)=H_\ovc(\ovy))]
$$
converges to $0$ faster than every $1/poly(n)$,
where the probability is taken over
\begin{itemize}
\item 
uniform choices of $\ovc\in (C_{p,n})^m$,
\item 
the coin tosses of $\CA$.
\end{itemize}
If this condition is not satisfied, then there is a PPT algorithm $\CA$
and $c,d>0$ satisfying $P_\CA(n) \ge c\cdot n^{-d}$ for infinitely many
indices $n$. Observe that
\begin{align*}
H_\ovc(\ovx)=H_\ovc(\ovy)
&\ \ \Leftrightarrow\ \ 
\sum_{j=1}^m x_j \ovc_j = \sum_{j=1}^m y_j \ovc_j\\
&\ \ \Leftrightarrow\ \ 
\sum_{j=1}^m (2+x_j-y_j) \ovc_j = \ovo,
\ \ \mbox{where } 2+x_j-y_j\in\{1,2,3\}
\\
&\ \ \Rightarrow\ \ 
\{z_j=(\ovo,(2+x_j-y_j)^{-1})\}_{j=1}^m
\mbox{ satisfies \eqref{eq:CSE-123}.}
\end{align*}
Thus, if we can efficiently find collisions for uniformly distributed 
$0/1$-spherical equations $H_\ovc\in \CH_n$ for infinitely many
indices $n$, then
we can efficiently find solutions for uniformly 
distributed $\CISE_{\{1,2,3\}}$.
Then, by Theorem \ref{th:CISE-hard}, there is a PPT algorithm 
solving $\SIVP_{\gamma=pn^6}$ with overwhelming probability 
on every lattice for infinitely many dimensions $n$, 
which contradicts the assumption \eqref{eq:SIVP-hardness}. 
Contradiction.
\end{proof}

\section{The acyclic graph word problem: average-case hardness}
\label{se:agwp}

$\CISE$ can be naturally reduced to different knapsack-type problems
in groups. One such reduction is discussed in this section.
The \emph{acyclic graph word problem} was introduced in 
\cite{Frenkel-Nikolaev-Ushakov:2014} as a convenient
generalization of $\SSP$ and as a tool for studying $\SSP$.
Let $X$ be a generating set for $G$.

\medskip\noindent
\textbf{The acyclic graph word problem, $\AGP(G,X)$:}
Given an acyclic directed graph $\Gamma$ with edges labeled by letters 
in $X\cup X^{-1}\cup \{\varepsilon\}$ with two marked vertices, $\alpha$ and $\omega$, decide whether there is an oriented path in $\Gamma$ from $\alpha$ to $\omega$ labeled by a word $w$ such that $w=1$ in $G$.

\medskip
It is easy to show that complexity of $\AGP$ does not depend on a choice
of a generating set for $G$ and the problem can be abbreviated $\AGP(G)$.
Also, we can work with edges labeled with words over the alphabet $X$.
There is a natural polynomial-time reduction from (decision/search)-$\CISE$ 
to (decision/search)-$\AGP$. Indeed, for an instance of $\CISE$
$$
\left\{
\begin{array}{l}
\prod_{i=1}^m z_i^{-1} c_i z_i= c\\
z_i\in Z_i,
\end{array}
\right.
$$
we can construct a labelled digraph $\Gamma=(V,E)$, where
\begin{itemize}
\item 
$V=\{0,1,\dots,m,m+1\}$
\item 
$E=\Set{j-1\stackrel{z^{-1}c_jz}{\longrightarrow}j}{z\in Z_j,\ j=1,\dots,m} 
\cup \{m\stackrel{c^{-1}}{\longrightarrow}m+1\}$
\item 
the starting point $\alpha$ is $0$,
\item 
the terminal point $\omega$ is $m+1$.
\end{itemize}
The construction is visualized in Figure \ref{fi:cise-to-agwp}.
By construction of $\Gamma$, the following statement holds.

\begin{figure}[h]
\begin{center}
\includegraphics[scale=0.75]{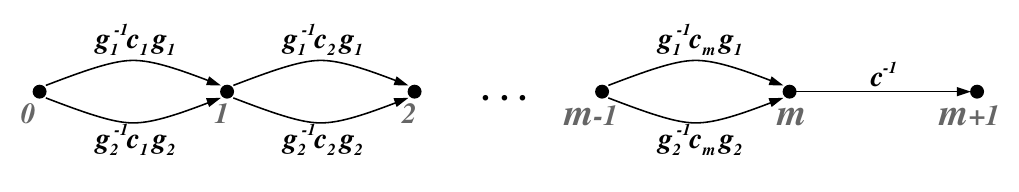}
\end{center}
\caption{An instance of $\AGP$ corresponding to an instance of $\CISE$
(assuming that $Z_j=\{g_1,g_2\}$).}
\label{fi:cise-to-agwp}
\end{figure}

\begin{prop}\label{pr:cise-to-agwp}
The instance of $\CISE$ is positive
if and only if the instance of $\AGP$ is positive.
\end{prop}

The reduction described above 
induces a system of probability measures $\{\mu_n\}_{n\in\MN}$
on instances of $\AGP$ over the class of groups $\MZ_p^n\rtimes \MZ_p^\ast$.
In fact, each $\mu_n$ gives uniform distribution on a subset 
of instances of $\AGP$ of type shown in Figure \ref{pr:cise-to-agwp}.
That establishes a one-to one correspondence between
uniformly randomized $\CISE_{\{1,2\}}$ described in Section \ref{se:cise}
and $\AGP$ endowed with $\{\mu_n\}_{n\in\MN}$, which implies
average-case hardness of $\AGP$ (again, assuming \eqref{eq:SIVP-hardness}).

The converse reduction appears to be implausible.
In fact, for some groups the converse is clearly not true, 
e.g., for $\MZ_3^\omega$ 
(and for the class of groups $\{\MZ_3^n\}_{n\in\MN}$)
we have
\begin{itemize}
\item 
$\CISE$ is polynomial-time decidable, but
\item 
$\SSP$ (and hence $\AGP$) is $\NP$-complete.
\end{itemize}

\section{Open problems and questions for further study}

In this section we outline several problems related to 
spherical equations, 
constrained spherical equations,
relations between constrained/unconstrained problems,
average-case complexity, and design of hash functions.

\subsection{Spherical equations over finite groups: constrained versus unconstrained}

\begin{prob}\label{pr:pr1}
Does there exist a class of finite groups $\CG$ satisfying 
the following conditions:
\begin{itemize}
\item 
the Diophantine problem for spherical equations over $\CG$
is efficiently decidable,
\item 
the Diophantine problem for constrained
spherical equations is computationally hard?
\end{itemize}
\end{prob}

Denote by $\prm$ the set of all prime numbers.
It was proved in \cite{Mattes-Ushakov-Weiss:2024}
that the Diophantine problem for
spherical equations over $\CG=\{\GL(2,p)\}_{p \in \prm}$
can be solved efficiently.

\begin{prob}\label{pr:pr2}
Does $\{\GL(2,p)\}_{p \in \prm}$ satisfy the second condition of Problem \ref{pr:pr1}?
\end{prob}

Problem \ref{pr:pr2} can be generalized as follows.

\begin{prob}\label{pr:pr3}
Does any family of classical finite groups 
(such as $\SL(n,p)$, $\PSL(n,p)$, etc., see \cite{Kleidman-Liebeck:1990}) 
satisfy conditions of Problem \ref{pr:pr1}?
\end{prob}

\subsection{Constrained spherical equations: foundation of average-case hardness}
\label{se:CSE-foundations}

In Section \ref{se:cise} we showed that for  groups $\{G_{p,n}\}$ 
constrained spherical equations (randomized in a certain way)
are hard on average by establishing a one-to-one correspondence
between $\CISE_{\{1,2\}}$ and $\ISIS_{\{0,1\}}$.
Average-case hardness of $\ISIS_{\{0,1\}}$ is linked 
(by M. Ajtai) to the worst-case hardness 
of lattice approximation problems, such as $\SIVP_\gamma$,
which makes hardness of $\SIVP_\gamma$ the foundation of average hardness
for $\CISE_{\{1,2\}}$ and $\ISIS_{\{0,1\}}$.
Informally, is it possible to untie $\CISE_{\{1,2\}}$ and $\ISIS_{\{0,1\}}$ 
from $\SIVP_\gamma$?

One way to approach this question is to utilize
a discrete logarithm type self-reduction idea.
For $\ISIS_{\{0,1\}}$ that means to 
design a class of randomized self-reductions $\Phi$ between instances of
$\ISIS_{\{0,1\}}$ and use $\Phi$ to enhance $\ISIS$-solvers as follows.
If $A$ is an $\ISIS$-solver, then an enhanced $\ISIS$-solver $B$
for a given instance $I$  performs the following:
\begin{itemize}
\item[(a)]
Apply $A$ to $I$ and if it succeeds, then output the result.
\item[(b)]
Choose a random $\varphi\in\Phi$
and apply $A$ to $\varphi(I)$. If it succeeds and
produces a solution for $\varphi(I)$, then use it to ``reconstruct''
and output a solution for $I$.
\item[(c)]
Repeat (b) until it succeeds.
\end{itemize}
If $\Phi$ is sufficiently large, then this approach might work.
Notice that existence of uniform self-reductions for $\ISIS$
(like for the discrete logarithm problem)
appears to be implausible.

\begin{prob}\label{pr:pr4}
Design a class of randomized self-reductions $\Phi$ for $\ISIS_{\{0,1\}}$ and use them
to prove the following.
If $\ISIS_{\{0,1\}}$ can be solved by a PPT algorithm $A$ on a 
non-negligible set of instances, then it can be solved by a PPT algorithm 
$B$ on every instance with overwhelming probability.
\end{prob}

Similarly, it would be interesting to investigate self-reductions for $\CISE$
over some groups $\CG$
and use them to demonstrate that the worst-case hardness for $\CISE$
implies its average-case hardness.

\begin{prob}\label{pr:pr4a}
For a class of spherical equations over a class of groups $\CG$,
design a class of randomized self-reductions for $\CISE$ and 
use them to prove the following.
If $\CISE$ can be solved by a PPT algorithm $A$ on a 
non-negligible set of instances, then it can be solved by a PPT algorithm 
$B$ on every instance with overwhelming probability.
\end{prob}

\subsection{Constrained spherical equations: generalization}

As we have shown in Section \ref{se:agwp}, instances of $\CISE$ 
can be transformed into instances of $\AGP$ of certain type (shown in 
Figure \ref{pr:cise-to-agwp}).
Let us generalize the digraph shown in Figure \ref{pr:cise-to-agwp}
and modify it into a finite-state transducer.
For a table of elements $C=\{c_{ij}\}$ from $G$ (where $i=1,\dots,m$ and $j=0,1$)
define a transducer $\Gamma_C$ shown in Figure \ref{fi:transducer}.
\begin{figure}[h]
\begin{center}
\includegraphics[scale=0.75]{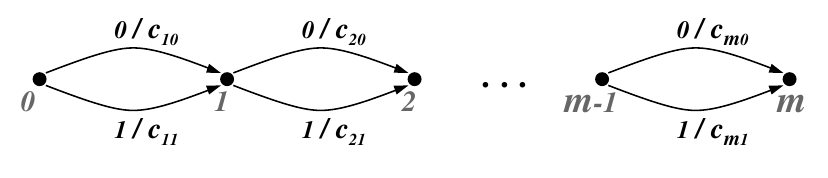}
\end{center}
\caption{The transducer $\Gamma_C$.}
\label{fi:transducer}
\end{figure}
Each $\Gamma_C$ defines a function $J_C\colon \{0,1\}^m\to G$ as follows:
$$
(b_1,\dots,b_m)
\ \ \stackrel{J_C}{\mapsto}\ \ 
c_{1b_1}
\cdots
c_{mb_m}.
$$
By design,
the problem of finding a $J_C$-preimage $(b_1,\dots,b_m)\in\{0,1\}^m$ of a given 
element $c$ generalizes $\CISE$.
Also, notice that the functions $J_C$ generalize
hash function construction introduced in \cite{Zemor:1991} 
and its numerous variations 
(e.g., \cite{Tillich-Zemor:1994,Bromberg-Shpilrain-Vdovina:2015}),
including those utilizing monoids.

\begin{prob}\label{pr:pr5}
Investigate average-case hardness of computing $J_C^{-1}$ for $C$ 
sampled uniformly randomly.
\end{prob}

The most promising approach to Problem \ref{pr:pr5} appears to be
via self-reductions similar to those discussed in Section \ref{se:CSE-foundations}.
Furthermore, it might be easier to design self-reductions for functions
$J_C$ because 
$J_C$ does not require pairs $c_{i0}$ and $c_{i1}$ to be conjugate.

\begin{prob}\label{pr:pr61}
For a class of groups $\CG=\{G_n\}$ design a randomized self-reduction
for the problem of computing $J_C^{-1}$ 
satisfying the following:
if $J_C^{-1}$ can be computed by a PPT algorithm on a non-negligible
set of instances, then it 
can be computed by a PPT algorithm on every instance 
with overwhelming probability.
\end{prob}

\begin{prob}\label{pr:pr6}
For a class of groups $\CG$
investigate cryptographic properties 
for the family of functions $\{J_C\}_{C\in G^{2\times m}}$.
\end{prob}

\subsection{Constrained spherical equations: learning without errors}

Consider a black-box device $D$ that computes values of a ``hidden'' function
$H_\ovc\colon \{0,1\}^m\to G$ for some $\ovc\in G^m$
hardwired inside of $D$. The device has a single button
which, when pressed, produces a pair $(\ovb,g)$, where
\begin{itemize}
\item 
$\ovb=(b_1,\dots,b_m)\in\{0,1\}^m$ is chosen uniformly randomly,
\item 
$g=H_\ovc(\ovb)$.
\end{itemize}
\emph{To learn the hidden function} means to find 
the secret element $\ovc$ using a number of sampled pairs $(\ovb,g)$.

\begin{prob}
Investigate computational complexity of learning $\ovc$ 
from a hidden function $H_\ovc$
for different (classes of) finite/infinite groups.
\end{prob}

In a similar way, we can consider a device $D$ that computes 
values of a ``hidden'' function $J_C\colon \{0,1\}^m\to G$
and the problem of learning $C$.

\begin{prob}
Investigate computational complexity of learning $C$ 
from a hidden function $J_C$
for different (classes of) finite/infinite groups.
\end{prob}

\subsection{Spherical equations over infinite groups: decidability}

For what (infinite) groups $G$ is the Diophantine problem for spherical equations decidable,
but the problem for constrained spherical equations is not?
Obviously any instance of $\CISE$ with a finite search space
(e.g. when $|G|<\infty$ or when each variable can attain finitely many values)
can be solved by enumerating all possible solutions. 
If $|G|=\infty$ and $Z_j$ are allowed to be infinite,
then $\CISE$ can become undecidable.
This question depends on how we allow to constrain variables, e.g.
\begin{itemize}
\item 
sets $Z_j$ can be finitely generated subgroups of $G$,
\item 
or (more generally) sets $Z_j$ can be defined by rational subsets of $G$.
\end{itemize}

\begin{prob}
For what (infinite) groups $G$ the Diophantine problem for spherical
equations is decidable, but spherical equations with 
constraints of the form $z_j\in Z_j$, where \emph{$Z_j$ is 
a finitely generated subgroup} of $G$ is not decidable. 
\end{prob}

Obvious candidates are groups containing a subgroup $H$
with undecidable membership problem, such as partially commutative groups.
We believe that for these groups the Diophantine problem for
spherical equations belongs to $\NP$, but the problem for
the constrained conjugacy equations
$$
\left\{
\begin{array}{l}
z^{-1} c z = c\\
z\in H,
\end{array}
\right.
$$
is not decidable. We also formulate a similar problem for
rationally constrained spherical equations.

\begin{prob}
For what (infinite) groups $G$ the Diophantine problem for spherical
equations is decidable, but spherical equations with 
constraints of the form $z_j\in Z_j$, where \emph{$Z_j$ is 
a rational subset} of $G$, is not decidable.
\end{prob}

Another class of groups with undecidable membership problem is the 
braid groups. For the braid groups it is not 
known how to approach even unconstrained spherical equations.

\begin{prob}
Is the Diophantine problem for spherical (quadratic)
equations over braid groups decidable?
\end{prob}

\subsection{Spherical equations over infinite groups: universality}

Recall that a set of functions $H=\{h\colon D\to R\}$ is \emph{universal} if 
for any distinct $x,y\in D$ 
$$
\Pr[h(x)=h(y)] \le \tfrac{1}{|R|},
$$
where the probability is taken over a uniformly chosen $h\in H$.
This condition can be relaxed by allowing
$\Pr[h(x)=h(y)]$ to be $O(\tfrac{1}{|R|})$.
Notice that in this paper 
for $R$ we have used certain classes of finite groups only
(namely, $\{\MZ_p^n \rtimes \MZ_p^\ast\}$ and $\{S_n\}$).
Is it possible to construct a family of efficient universal functions
using a single infinite group?

\begin{prob}
Design a family of universal 
$0/1$-spherical functions or $J_C$ functions with the ranges 
$R$ being a part of the same infinite group.
\end{prob}

%The first Grigorchuk group
%(see \cite{Grigorchuk:1980} and \cite[Chapter 8]{Harpe})
%and similar groups seem to be the best candidates for this problem
%as they naturally contain infinitely many subgroups of the form $\MZ_n^k$.

\subsection{Spherical functions: public-key encryption}

\begin{prob}
Is it possible to design a public-key encryption scheme which security
is based on computational hardness of solving constrained spherical equations
(or another related group-theoretic problem)?
\end{prob}

\bibliographystyle{plain}

\end{document}